\documentclass{amsart}
\usepackage{amssymb,amsthm,amsmath}
\usepackage{enumerate}
\usepackage{cite}
\usepackage[hidelinks]{hyperref}
\usepackage{dsfont}

\def \Nats {\mathds{N}}
\def \Ints {\mathds{Z}}

\DeclareMathOperator{\restrictedto}{\upharpoonright}

\newcommand{\tup}[1]{\langle #1\rangle}
\newcommand{\otn}[2]{#1_{1},\dots ,#1_{#2}}
\newcommand{\set}[1]{\{#1\}}
\newcommand{\setarg}[2]{\{#1\ |\ #2\}}
\newcommand{\setcol}[2]{\{#1 : #2\}}

\newtheorem*{introthm}{Theorem}
\newtheorem{lemma}{Lemma}[section]

\newtheorem{prop}[lemma]{Proposition}

\theoremstyle{definition}
\newtheorem{definition}[lemma]{Definition}
\newtheorem*{nonumdefinition}{Definition}

\newtheorem{observation}[lemma]{Observation}
\newtheorem*{nonumobservation}{Observation}
\newtheorem{notation}[lemma]{Notation}
\newtheorem{remark}[lemma]{Remark}
\newtheorem*{nonumremark}{Remark}
\newtheorem*{question}{Question}
\newtheorem*{nonumexample}{Example}

\DeclareMathOperator{\MaxCliques}{M}

\newcommand{\lang}[1]{\mathcal{L}_{#1}}
\newcommand{\predim}{\lambda}

\newcommand{\dims}{\Lambda}
\newcommand{\closure}[1]{\cl^{#1}}
\newcommand{\cardstar}[1]{|#1|_*}
\newcommand{\maxcliques}[1]{\MaxCliques(#1)}
\newcommand{\strong}{\leqslant}

\DeclareMathOperator{\clq}{clq}

\DeclareMathOperator{\cl}{cl}
\DeclareMathOperator{\pregeometry}{PG}
\def \Cclq {\mathcal{C}^{\clq}}

\def \Corig {\mathcal{C}}
\newcommand{\formula}[1]{\varphi_{T_{#1}}}
\newcommand{\reduct}[1]{#1^T}
\def\nsim{\raise.17ex\hbox{$\scriptstyle\sim$}}

\newcommand{\Mclq}{\mathbb{M}^{\clq}}

\newcommand{\Morig}{\mathbb{M}}

\def\Cc{\mathbb C}
\def\CL{\mathcal L}
\def\0{\emptyset}

\def \tupsize {r}
\def \arity {s}

\title{Reducts of {Hrushovski's} constructions of a higher geometrical arity}
\author[A. Hasson and O. Mermelstein]{Assaf Hasson and Omer Mermelstein}
\date{\today}
\address{Department of Mathematics, Ben Gurion University of the Negev \\ P.O.B 653, Be'er Sheva 8410501, Israel}
\email{hassonas@math.bgu.ac.il}
\email{omermerm@math.bgu.ac.il}
\thanks{The research was partially supported by ISF grant No. 181/16}
\begin{document}
\begin{abstract}
Let $\Morig_n$ denote the structure obtained from Hrushovski's (non collapsed) construction with an n-ary relation and $\pregeometry(\Morig_n)$ its associated pre-geometry. It was shown in \cite{DavidMarcoOne} that $\pregeometry(\Morig_3)\not\cong \pregeometry(\Morig_4)$. We show that $\Morig_3$  has a reduct, $\Mclq$ such that  $\pregeometry(\Morig_4)\cong\pregeometry(\Mclq)$. To achieve this we show that $\Mclq$ is a slightly generalised Fra\"iss\'e-Hrushovski  limit incorporating into the construction non-eliminable imaginary sorts in $\Mclq$.  
\end{abstract}
\maketitle
\section{Introduction}
The class of combinatorial pregeometries associated with a structure (or a theory) is an important invariant in geometric stability theory and its bifurcations, going back to Baldwin-Lachlan \cite{BaldLach}, Zilber's Trichotomy conjecture \cite{ZilberTrichotomy} and its many applications, e.g. \cite{HrushovskiZilber-Zariskigeometries}, Shelah's analysis of super-stable theories \cite[Chapters V, IX, X]{Shelah}
, \cite{Hrushovski-Locallymodularregulartypes} and more. 

In the late 1970s Zilber suggested a classification of the geometries associated with strongly minimal theories based on the algebraic structures interpreted by those theories: trivial if no group is interpretable, projective if the theory is 1-based but non-trivial, or the geometry associated with an algebraically closed field otherwise. 

It follows from the fundamental theorem of projective geometry (e.g, \cite[\S II]{ArtinGeometry}) that (infinite) projective geometries are in one-to-one correspondence with projective spaces, and are classified by their (possibly non-commutative) fields of scalars. Though there is no simple characterisation of the geometries associated with algebraically closed fields, those are characterised by the characteristic of the field, \cite{EvHr}. 

Combined with \cite{Hrushovski-Locallymodularregulartypes} and \cite{Rabinovich} (and more generally, \cite{HaSu}) these observation imply that if $T$ is strongly minimal whose geometry is either that of an algebraically closed field or locally modular then the geometries of reducts of $T$ are partially ordered.  Namely, if $T'$ is a reduct of $T$ as above and $\pregeometry(T)\not\cong\pregeometry(T')$ then the geometry of $T'$ is strictly coarser than the geometry of $T$. 
Thus, if $T$ is the theory of an algebraic curve over an algebraically closed field, $K$ and $T'$ is a reduct of $T$ then either $T'$ interprets an algebraically closed field, in which case $T$ and $T'$ have isomorphic geometries, or $T'$ is locally modular. In the latter case a reduct $T''$ of $T'$ is either trivial, or $T'$ and $T''$ have the geometries of projective spaces over fields $F'\ge F''$ (respectively), \cite[Corollary 4.5.9]{PillayBook}

Similar results can be obtained for reducts of o-minimal structures (for o-minimal structures this is an immediate consequence of the o-minimal Trichotomy Theorem, \cite{PeStTricho}, for the more general claim, this follows from \cite{HaOnPe}). 

In the present note we show that the above is not true in Hrushovski's (non-collapsed) construction. Let $\Morig_n$ denote the generic structure associated with Hrushovski's construction in the language consisting of a single $n$-ary relation (see Section \ref{reduct} for precise definitions) and $\pregeometry(\Morig_n)$ the pre-geometry associated with its unique regular type of rank $\omega$. Ferreira and Evans show, \cite{DavidMarcoOne}, that $\pregeometry(\Morig_n)\cong \pregeometry(\Morig_m)$ if and only if $n=m$, and the same remains true even locally. It is an easy exercise to verify that if  $n<m$ we can, possibly naming finitely many parameters, identify $\Morig_n$ with a (proper) reduct of $\Morig_m$. E.g., to find $\Morig_3$ in $\Morig_4 $ we can consider the reduct to the relation $S(x,y,z):=R(x,x,y,z)$ (assuming $R$ allows repetitions). In case $R$ does not allow repetitions we can, fixing a generic $a$, consider the reduct $S(x,y,z):=R(a,x,y,z)$. 

Keeping the above notation, our main result is: 
\begin{introthm}
	Let $n>2$ be a natural number, $0<r<n$ and $s=n-r+1$. Then there exists a reduct $\Mclq$ of $\Morig_n$ such that: 
	\begin{enumerate}
		\item $\Mclq$ is a Fra\"iss\'e-Hrushovski limit with respect to a pre-dimension function allowing a polynomial number of relations. 
		\item $\pregeometry(\Mclq)\cong \pregeometry(\Morig_{rs})$. 
	\end{enumerate}
\end{introthm}

This result implies, in particular, that the geometries associated with Hrushovski constructions are not linearly ordered by reducts. Indeed, consider $\Morig_4$ and identify $\Morig_3$ as a (proper) reduct thereof. By our main theorem (with $n=3, r=s=2)$) there is a (proper) reduct $\Mclq$ of $\Morig_3$ such that $\pregeometry(\Mclq)\cong \pregeometry(\Morig_4)$. Finally, $\pregeometry(\Morig_4)\not \cong \pregeometry(\Morig_3)$ by  \cite{DavidMarcoOne}. 

In \cite{EvTrivial} Evans gives an example of a theory $T$ which is, in the terminology of \cite{JBGoode}, trivial for freedom (namely, if $a,b,c$ are pairwise independent over a parameter set $A$ then $a$ and $b$ are independent over $Ac$) and $1$-based with a reduct isomorphic to $\Morig_n$ (so neither trivial nor 1-based). The example of \cite{EvTrivial} is, however, strictly stable and therefore not of a geometric nature. Our example takes place in an $\omega$-stable theory, and the geometry we study is precisely that of the unique regular type of rank $\omega$. In view of the results of \cite{DavidMarcoTwo} there is good reason to believe that the results of the present paper can be reproduced in the context of strongly minimal theories. Although we see no obvious obstacles, the technicalities to sort out seem to go beyond the scope of the present note.  

Finally, we remark that though, in the notation of the main theorem,  $\pregeometry(\Mclq)\cong \pregeometry(\Morig_{rs})$, the two structures are not isomorphic (or even elementarily equivalent). Thus, it does not follow from our result that $\Mclq$ itself has a reduct whose geometry is isomorphic to that of $\Morig_k$ for some $k>rs$. The following remains open: 

\begin{question}
	Let $\Morig_n$ be the generic of Hrushovski's (non-collapsed) ab inito construction with a single $n$-ary relation. Is there, for all $k>n$, a reduct $\mathcal M(k)$ of $\Morig_n$ such that $\pregeometry(\mathcal M(k))\cong \pregeometry(\Morig_k)$? 
\end{question}

%

\subsection{Preliminaries}

A category whose objects form a class of finite (relational) structures $\mathbb{C}$, closed under isomorphisms and substructures, and whose morphisms, $\strong$, are (not necessarily all) embeddings, is an \emph{amalgamation class} (or has the \emph{Amalgamation property} and \emph{Joint Embedding property}) if:
\begin{itemize}
\item [(AP)]
If $A,B_1,B_2\in \mathbb{C}$ are such that $A\strong B_1,B_2$, then there exists some $D\in\mathbb{C}$ and embeddings $f_i:B_i \to D$ such that $f_i[B_i]\strong D$, $f_1\restrictedto A = f_2\restrictedto A$, and $f_1[A]\strong D$.
\item [(JEP)]
If $A_1,A_2\in \mathbb{C}$, then there exists some $B\in\mathbb{C}$ and embeddings $f_i:A_i\to B$ such that $f_i[A_i]\strong B$ for $i=1,2$.
\end{itemize}
By Fra\"iss\'e's Theorem, to every amalgamation class with countably many isomorphism types is associated a unique (up to isomorphism) countable structure $\mathbb{M}$ satisfying
\begin{enumerate}
\item
Every finite substructure of $\mathbb{M}$ is an element of $\mathbb{C}$.
\item
Whenever $A\strong \mathbb{M}$ and $A\strong D\in\mathbb{C}$, there is an embedding $f:D\to \mathbb{M}$ fixing $A$ pointwise such that $f[D]\strong\mathbb{M}$.
\end{enumerate}
We call $\mathbb{M}$ a generic structure for $\mathbb{C}$.

Hrushovski showed that if $\CL$ is a countable finite relational language\footnote{This generalises easily to infinite languages, provided every finite structure supports only finitely many non-empty relations.}, and for a finite $\CL$-structure $A$ we let $\delta(A):=|A|-k(A)$, where $k(A)$ is the number of $\CL$-relations in (powers of) $A$, then the class $\Cc$ of all finite $\CL$-structures 
$A$ such that $\delta(B)\ge 0 $ for all $B\subseteq A$ is an amalgamation class with respect to a class of, so called, self-sufficient (or strong) embeddings. Provided the language contains at least one $n$-ary relation for $n\ge 3$ or two binary relations, the $\Cc$-generic structure $\mathbb M$ is $\omega$-stable with a unique non-trivial regular type, $p_\Cc$. We call (the pregeometry of) $p_\Cc$ the pregeometry of $\mathbb M$. 

In the present note (Section \ref{construction}) we show that Hrushovski's construction can be carried out in a similar way if, e.g., in the context of a unique $s$-ary relation (possibly on $r$-tuples, rather than singletons), instead of defining $\delta(A):=|A|-k(A)$ we let $\lambda(A):=|A|-\sum_K (|K|-\arity+1)$ where the sum ranges over all maximal (large enough) \emph{cliques} in $A$, allowing $A^r$ to support a uniformly bounded polynomial (rather than linear) number of relations. A clique $K$ is determined by any $(s-1)$-tuple of its members, and thus can be viewed as an equivalence relation on $(s-1)$-tuples. In model theoretic terms cliques can be viewed as an imaginary sort that need not be eliminable in the generic structure. From that perspective, the above pre-dimension function is merely the straightforward adaptation of Hrufshovski's original pre-dimension function to this two sorted structure, and the non-linear number of relations is an artifact of counting (wrongly) the number of relations in definable congruence classes (of $r$-tuples) -- namely, cliques.

In Section \ref{reductSubsection} we show that the generic structure associated with the clique construction is isomorphic to a (proper) reduct of Hrushovski's original construction. From the point of view described above, this reduct, unlike the original construction, does not admit geometric elimination of imaginaries, and our construction can be viewed as a first step towards generalising the construction to incorporate imaginary elements. It is, apparently, a necessary step in classifying all reducts of Hrushovski's ab initio construction. We do not know, at this stage, how to identify all imaginaries associated with an arbitrary such reduct. 

As a test case we suggest the following: Let
\[
T(x_1, x_2, x_3, x_4) := \exists y_1, y_2, y_3 R(y_1, y_2, y_3)\wedge R(y_1, x_1, x_2)\wedge R(y_2, x_2, x_3) \wedge R(y_3, x_3, x_4)
\]
and let $M_T$ be the reduct of $\mathbb M_3$ whose unique atomic relation is $T$. Is $M_T$ an ab initio structure? With respect to what pre-dimension function? What is its geometry? 

Finally, in Section \ref{pregeometrySubsection} we conclude the proof of our main theorem. 

%


\section{The construction}\label{construction}

\def \rel {T}

Fix some natural $n\geq 2$ and $0<\tupsize<n$, and denote $\arity=n-\tupsize+1$. Let $\lang{}=\setcol{\rel_k}{k\geq \arity}$ be the relational language where $\rel_k$ is of arity $\tupsize\cdot k$. Throughout, we think of $\rel_k$ as a relation on $k$ distinct $\tupsize$-tuples -- condition (1) below. Additionally, we assume (2) and (3): 
\begin{enumerate}
\item
$\rel_k(\bar{x}_1,\bar{x}_2,\dots,\bar{x}_k)\implies \bigwedge_{i<j\leq k}\bar{x}_i\neq \bar{x}_j$
\item
$\rel_k(\bar{x}_1,\bar{x}_2,\dots,\bar{x}_k)\implies \bigwedge_{\sigma\in S_k} \rel_k(\bar{x}_{\sigma(1)},\bar{x}_{\sigma(2)},\dots,\bar{x}_{\sigma(k)})$
\item
$\rel_k(\bar{x}_1,\bar{x}_2,\dots,\bar{x}_k)\implies \bigwedge_{s\le i<k} \rel_i(\bar{x}_1,\bar{x}_2,\dots,\bar{x}_i)$
\end{enumerate}

\begin{definition}

\begin{enumerate}
\item 
For an $\lang{}$-structure $A$, we say that $K\subseteq A^{\tupsize}$ with $|K|\geq \arity$ is a \emph{clique} if for all $k\geq \arity$, whenever $\bar{x}_1,\dots,\bar{x}_k\in K$ are distinct, then $(\bar{x}_1,\dots,\bar{x}_k)\in \rel_k^A$.
\\We say that $K$ is a \emph{maximal} clique if there is no clique $K'\subseteq A^{\tupsize}$ such that $K'\supset K$.
For an $\lang{}$-structure $A$, define $\maxcliques{A}$ to be the set of maximal cliques of $A$.
\item
\label{cliques}
Define $\Cclq_0$ to be the class of finite $\lang{}$-structures $A$ such that whenever $K_1,K_2\in\maxcliques{A}$ are distinct, then $|K_1\cap K_2| < \arity$.
\end{enumerate}
\end{definition}

\begin{remark}\label{ReconstructFromCliques}
The language $\lang{}$ has relations of infinitely many arities, so that for any $M$, anti-chain under inclusion of sets of cardinality at least $s$, there exists a unique $\lang{}$-structure $A$ with $\maxcliques{A} = M$. Specifically, $A$ can be interpreted in the obvious way in the two sorted structure $(\bar A, M, I)$ where $\bar A = \bigcup M$ is the universe of $A$ and $I\subseteq \bar A^r\times \maxcliques{A}$ is interpreted as inclusion.

For structures $A\in\Cclq_0$, due to the restriction on intersections of cliques, the language $\lang{} = \set{T_{\arity}, T_{\arity+1}}$ suffices for this purpose ($T_s$ itself would not suffice: take $s+1$ pairwise distinct cliques, every two of which intersecting in a set of size $s-1$, such that their union, $A$, has size $s+1$. Then $A$ is not contained in any maximal clique, but this cannot be discerned using $T_s$ alone). The requirement $|K_1\cap K_2| < \arity$ is put in place in order to simplify the construction to come, e.g., it implies Observation \ref{uniqueExtensionObs}. For the purpose of getting the guarantee of the previous paragraph from a finite part of $\lang{}$, any uniform bound on $|K_1\cap K_2|$ will have sufficed.
\end{remark}

%
%

\begin{observation}
For an $\lang{}$-structure $A$ and a substructure $B\subseteq A$,
\[
\maxcliques{B} = \setcol{K\cap B^{\tupsize}}{K\in\maxcliques{A}, |K\cap B^{\tupsize}|\geq \arity}.
\]
\end{observation}


\begin{notation}
For a finite set $X$ denote $\cardstar{X} = \max\set{0,|X|-(\arity - 1)}$
\end{notation}

\begin{definition}
For every finite $\lang{}$-structure $A$ define
\[
t(A) = \sum_{K\in\maxcliques{A}} \cardstar{K}
\]
and
\[
\predim(A) = |A| - t(A).
\]
\end{definition}


\begin{observation}
\label{uniqueExtensionObs}
For $A\in\Cclq_0$, whenever $B\subseteq A$ and $K\in\maxcliques{B}$, there is a unique extension of $K$ to a maximal clique of $A$. In particular,
\[
t(B) =\sum_{K\in\maxcliques{A}}\cardstar{K\cap B}
\]
\end{observation}

\begin{lemma}
\label{submodularity of predim on clq}
The function $\predim:\Cclq_0\to \Ints$ is submodular. That is, 
\[
\predim(A\cup B) + \predim(A\cap B) \leq \predim(A) + \predim(B).
\]
whenever $D\in\Cclq_0$ and $A,B,A\cup B,A\cap B\subseteq D$ are induced substructures. 

\end{lemma}

\begin{proof}
For each $K\in\maxcliques{A\cup B}$ let $K_A,K_B,K_{AB}$ denote $K\cap A^{\tupsize} , K\cap B^{\tupsize}, K\cap (A\cap B)^{\tupsize}$ respectively. Observe that 
\[
\cardstar{K} + \cardstar{K_{AB}} \geq \cardstar{K_A} + \cardstar{K_B}
\]
for each $K\in\maxcliques{A\cup B}$. 
Thus, by Observation \ref{uniqueExtensionObs},
\begin{align*}
t(A\cup B) + t(A\cap B)
&= \sum_{K\in \maxcliques{A\cup B}} \cardstar{K} + \sum_{K\in \maxcliques{A\cup B}} \cardstar{K_{AB}}
\\
&\geq \sum_{K\in\maxcliques{A\cup B}}\cardstar{K_A} + \sum_{K\in\maxcliques{A\cup B}}\cardstar{K_B}
\\
&= t(A) + t(B),
\end{align*}
proving the statement.
\end{proof}

For a finite $\lang{}$-structure $A$ and a substructure $B\subseteq A$ define $\predim(A/B) = \predim(A\cup B) - \predim(B)$. Extend this definition to an infinite $\lang{}$-structure $A$ and a substructure $B$ by defining $\predim(A/B) = \inf\setcol{\predim(X/X\cap B)}{X\subseteq A, |X|<\infty}$. The definitions coincide on finite structures, by submodularity. Write $B\strong A$ if $\predim(X/B)\geq 0$ for every $B\subseteq X\subseteq A$. By submodularity, again, the relation $\strong$ is transitive.

\begin{definition}
Define $\Cclq$ to be the class of $\lang{}$-structures $A\in\Cclq_0$ such that $\emptyset\strong A$.
\end{definition}

\begin{remark}
One can define an analogue of $\Cclq$, where distinct cliques are allowed to intersect in arbitrarily large sets, and the only requirement of an $\lang{}$-structure $A$ is that $\emptyset\strong A$. In that case, the guarantee of a sufficient finite language of the second paragraph of remark \ref{ReconstructFromCliques} may not hold, seemingly. However, it can be shown that there is some finite $k=k(\tupsize,\arity)$ such that $\set{T_{\arity},\dots,T_k}$ does suffice. The proof is not entirely trivial, and we omit it. 
\end{remark}

\begin{definition}
Let $A_1,A_2\in \Cclq_0$ and let $B=A_1\cap A_2$ be a common induced substructure. Define the \emph{standard amalgam} of $A_1$ and $A_2$ over $B$ to be the unique $\lang{}$-structure $D$ whose universe is $A_1\cup A_2$ such that $\maxcliques{D} = M\cup M'$ where
\begin{gather*}
M = \setcol{K\in\maxcliques{A_1}\cup \maxcliques{A_2}}{|K\cap B^{\tupsize}|< \arity}
\\
M' = \setcol{K_1\cup K_2}{K_1\in\maxcliques{A_1}, K_2\in\maxcliques{A_2}, |K_1\cap K_2| \geq \arity}.
\end{gather*}
\end{definition}

\begin{observation}
Let $A_1,A_2\in \Cclq$ be such that $B = A_1\cap A_2$ is a common substructure. Let $D$ be the standard amalgam of $A_1$ and $A_2$ over $B$. Then $\predim(D/A_1) = \predim(A_2/B)$.
\end{observation}

For $A,B\in \Cclq_0$, say that an embedding of $\mathcal L$-structures $f:A\to B$ is \emph{strong} if $f[A]\strong B$. The class $\Cclq$ is closed under taking substructures and has the JEP. By the above observation, $\Cclq$ also has AP with respect to strong embeddings. Since $\Cclq_0$ has countably many isomorphism types it follows from  Fra\"iss\'e's Theorem that it has a unique countable generic structure $\Mclq$ defined by the property:
\begin{itemize}
\item[$(*)$]
Whenever $A\strong B\in \Cclq$ and $A\strong \Mclq$, there exists a strong embedding $f:B\to \Mclq$ fixing $A$ pointwise.
\end{itemize}

As was pointed out to us by the referee, if $r=1$ the above construction can be easily viewed as a standard Hrushovski's construction of a bi-partite graph. Namely, work in a two sorted language with sorts $P$ for points and $C$ for cliques and a unique relation, $R\subseteq P\times C$. If cliques are given weight $(s-1)$ then the standard pre-dimension function associated with such graphs is $\lambda(P,C) = |P| + (s-1)|C| - |R(P,C)|$ -- precisely the pre-dimension function associated with $\Mclq$ in case $r=1$ and provided each clique is related to at least $s$ points.

In case $r>1$ some adaptations are needed for a similar construction to work. E.g., add a binary relation $I$ on $P$ (to be interpreted as ``the domains of $\bar x$ and $\bar y$ intersect''). In that case points are precisely maximal $I$-cliques. If we require that elements in $P$ are sets of size $r$, rather than $r$-tuples\footnote{Allowing several relations in the construction this can be achieved by decomposing the relation $R$.},  then the restriction on the structure should be that every $p\in P$ belongs to exactly $r$ maximal $I$-cliques. 

We hope that a deeper study of this construction and its many possible variants may shed new light on the inner structure of Hrushovski's constructions and their reducts. 

\section{Relation to Hrushovski's ab initio construction}\label{reduct}

%
%
%
%

As above, we fix some natural $n\geq 2$ and $0<\tupsize<n$, and denote $\arity=n-\tupsize+1$.
Let $\lang{n}=\set{R}$ be the language of a single $n$-ary relation. For a finite $\lang{n}$-structure $A$, define $\delta(A) = |A| - |R^A|$. Define $\strong$ for $\lang{n}$-structures as defined above with resepct to $\lambda$. Let $\Corig_n$ be the class of finite $\lang{n}$-structures $A$ with $\emptyset\strong A$. It is closed under substructures and free amalgamation, and thus is an amalgamation class. Let $\Morig_n$ be the generic structure for the class $\Corig_n$.

\subsection{$\Mclq$ is a proper reduct of $\Morig_n$}
\label{reductSubsection}

For every natural number $k\geq \arity$ we let $\formula{k}(\bar{x}_1,\dots,\bar{x}_k)$, where $|\bar{x}_i| = \tupsize$, be the $\lang{n}$-formula specifying that there exists an $(\arity - 1)$-tuple $\bar{y}$ whose elements are distinct from the elements of $\bar{x}_1,\dots,\bar{x}_k$ such that, denoting by
$X$ the set of all elements appearing in $\bar{x}_1,\dots,\bar{x}_k,\bar{y}$,
\begin{itemize}
	\item
	$\bar{x}_i\neq\bar{x}_j$ for all $1\leq i<j\leq k$
	\item
	$R^{X} = \setarg{(\bar{y},\bar{x}_i)}{1\leq i\leq k}$
	\item
	$X\strong B$ in any superstructure $B$ (in the ambient structure) with $|B|\leq |X|+\arity$.
\end{itemize}
Note that the last item of the above list guarantees that $\bar{y}$ is unique. For every $\lang{n}$-structure $A$, denote by $\reduct{A}$ the reduct $\tup{A,\formula{2}(A),\formula{3}(A),\dots}$. In the present subsection we prove: 

\begin{prop}\label{reductProp}
	The structure $\Mclq$ is isomorphic to $\Morig_n^T$. 
\end{prop}
\smallskip
By \cite[3.6.7,3.6.9]{Mermelstein2016}, in order to show that $\reduct{\Morig_n}\cong \Mclq$, it suffices to show:
\begin{enumerate}
\item
If $A\in \Corig_n$ then $\reduct{A}\in\Cclq$.
\item
Whenever $M$ is an $\lang{n}$-structure, $A\in\Corig_n$ and $A\strong M$, the substructure induced on the set $A$ by $\reduct{M}$ is exactly $\reduct{A}$.
\item
For every $A\in\Corig_n$ and $B_c\in\Cclq$ such that $\reduct{A}\strong B_c$, there exists some $C\in\Corig_n$ such that $A\strong C$ and $B_c\strong \reduct{C}$.
\end{enumerate}
Showing also that
\begin{enumerate}
\item[$(4)$]
For any $F\in\Corig_n$, there exist $A,B\in\Corig_n$ with $F\strong A,B$ such that $A,B$ are not isomorphic over $F$, but $\reduct{A},\reduct{B}$ are isomorphic over $F$.
\end{enumerate}
will prove that $R^{\Morig_n}$ is not definable in $\reduct{\Morig_n}$.

We prove the statements in order.
\begin{proof}[Proof of $(1)$]
Let $A\in \Corig_n$. Denote $A_c = \reduct{A}$. Clearly, $A_c\in \Cclq_0$. Let $B\subseteq A$ be an arbitrary nonempty substructure of $A$. Denote by $B_c$ the substructure induced on $B$ by $A_c$. Consider
\[
\bar{B} = B\cup\setarg{\bar{y}\in A}{\setcol{\bar{x}}{B\models R(\bar{y},\bar{x})}\in\maxcliques{B_c}}
\]
as a substructure of $A$. Then
\begin{align*}
0\le \delta(\bar{B}) &= (|B| + |\bar{B}\setminus B|) - |R^{\bar{B}}|
\\
&\leq |B| + (\arity-1)\cdot|\maxcliques{B_c}| - \sum_{K\in\maxcliques{B_c}} |K|
\\
& = |B| - \sum_{K\in\maxcliques{B_c}} \cardstar{K}
\\
&= \predim(B_c)
\end{align*}
Thus, $\predim(B_c)\geq 0$ and $A_c\in \Cclq$.
\end{proof}

\begin{proof}[Proof of $(2)$]
Let $M$, $A$ be $\lang{n}$-structures, $A\in\Corig_n$ and $A\strong M$ .
Let $(\bar{a}_1,\dots,\bar{a}_k)\in  \formula{k}(A^{\tupsize})$ and assume $\bar{y}\in M$ is such that $M\models\bigwedge_{1\leq i\leq k} R(\bar{y},\bar{a}_i)$. It must be that $\bar{y}\in A^{\arity-1}$, for otherwise, by definition $k\ge \arity$, and $\delta(\bar{y}/A) < 0$ in contradiction to $A\strong M$.

Also by $A\strong M$, if $\set{\bar{a}_1,\dots,\bar{a}_k,\bar{y}}\not\strong B$ where $B\subseteq M$, then already within $A$ we have $\bar{a}_1\dots\bar{a}_k\bar{y}\not\strong B\cap A$, with $|B\cap A|\leq |B|$.

Thus, $M\models \formula{k}(\bar{a}_1,\dots,\bar{a}_k)$ if and only if $A\models\formula{k}(\bar{a}_1,\dots,\bar{a}_k)$.
\end{proof}

\begin{proof}[Proof of $(3)$]
Let $A\in \Corig_n$, $B_c\in\Cclq$ be such that $\reduct{A}\strong B_c$. Denote $A_c = \reduct{A}$, by $(1)$ we know $A_c\in \Cclq$. Denote by $B$ the underlying set of $B_c$.

For each $K\in \maxcliques{A_c}$ let $\bar{y}_K\in A^{\arity-1}$ be the unique tuple such that $\setarg{\bar{a}\in A^{\tupsize}}{(\bar{y}_K,\bar{a})\in R^A} = K$, and let $\widehat{K}\in \maxcliques{B_c}$ be the unique maximal clique in $B_c$ extending  $K$. Let
\[
R_0 = \bigcup_{K\in \maxcliques{A_c}} \setarg{(\bar{y}_K,\bar{b})}{\bar{b}\in \widehat{K}\setminus K}
\]

For each $L\in \maxcliques{B_c}$ such that $L\cap A^{\tupsize}\notin \maxcliques{A_c}$, let $\bar{z}_L$ be an $(\arity -1)$-tuple of new elements. Let
\[
R_1:=\{(\bar z_L, \bar b): \bar b\in L, L\in \maxcliques{B_c}, L\cap A^r\notin\maxcliques{A_c}\}
\]
and 
\[
	Z=\{\bar z: L\in \maxcliques{B_c}, L\cap A^r\notin \maxcliques{A_c}\}. 
\]

Define $C$ to be the $\lang{n}$-structure with underlying set $B\cup Z$ and
\[
R^C = R^A\cup R_0 \cup R_1
\]
and denote $C_c = \reduct{C}$. Clearly $B_c$ is a substructure of $C_c$. Moreover, $B_c\le A_c$. Indeed, as $\maxcliques{C_c} = \maxcliques{B_c}$, we get, by construction,  that $\predim(B\cup Z_0/B) = |Z_0|$ for any $Z_0\subseteq Z$.

It remains to show that $A\strong C$, i.e., that $\delta(X/A)\geq 0$ for any $A\subseteq X \subseteq C$. Note that if $L\in\maxcliques{B_c}$ with $L\cap A^{\tupsize}\notin \maxcliques{A_c}$, then $\delta(X\bar{z}_L/A)<{\delta((X\setminus \bar{z}_L)/A)}$ if and only if $|L\cap X^{\tupsize}|\geq \arity$. Thus, it will suffice to prove the inequality under the assumption that for all such $L$, $\bar{z}_L\in X^{(\arity-1)}$ if and only if $|L\cap X^{\tupsize}| \geq \arity$. Then
\begin{align*}
\delta(X/A) &= (|X\setminus(X\cap B)| +|(X\cap B)\setminus A|) - (|R_0\cap X^n| + |R_1\cap X^n|)
\\
&\geq |(X\cap B)\setminus A| - \sum_{K\in\maxcliques{A_c}} |(\widehat{K}\cap X^{\tupsize})\setminus K| - \sum_{\substack{L\in \maxcliques{B_c} \\ |L\cap X^{\tupsize}|\geq \arity \\ L\cap A^{\tupsize}\notin\maxcliques{A_c}}} \cardstar{L\cap X^{\tupsize}}
\\
&= \predim(X\cap B_c/A_c) \geq 0
\end{align*}
\end{proof}

\begin{proof}[Proof of $(4)$]
Let $F\in \Corig_n$. Define $A$, $B$ to be the $\lang{n}$-structures with underlying set $F\bar{a}$, where $\bar{a}$ is an $n$-tuple of new elements, and
\begin{align*}
R^A &= R^F
\\
R^B &= R^F\cup\set{\bar{a}}
\end{align*}
\end{proof}

This finishes the proof of Proposition \ref{reductProp}. We proceed to studying the geometry of $\Mclq$. 

\subsection{The pregeometry of $\Mclq$}
\label{pregeometrySubsection}
Recall that for an $\lang{}$-structure $A$ and some substructure $B\subseteq A$ 
\[
\closure{A}(B) = \bigcup\setarg{X\subseteq A}{\predim(X/X\cap B)\leq 0}.
\]
is a closure operator giving rise to a pregeometry on the underlying set of $A$. The dimension function associated to this closure operator is
\[
\dims(B) = \min\setcol{|X|}{X\subseteq B,~\closure{A}(X) = \closure{A}(B)}
\]
and we say that $B$ is independent in $A$ if $\dims(B) = |B|$. A pregeometry is defined in a similar way on $\lang{k}$-structures, for any natural $k$. For an $\lang{}$-structure or an $\lang{k}$-structure $A$, we denote its associated pregeometry by $\pregeometry(A)$.

A pregeometry is uniquely determined by any one of the following: its closure operator on finite sets, its dimension function, its collection of independent subsets. We say that two pregeometries are isomorphic if there exists a bijection between the two, preserving any one of these in both directions.

For amalgamation classes $(\mathcal{D}_1,\strong), (\mathcal{D}_2,\strong)$ of either $\lang{}$-structures or $\lang{k}$-structures, write $\mathcal{D}_1\overset{*}{\rightsquigarrow}\mathcal{D}_2$ if
\begin{itemize}
\item[$(*)$]
Whenever $A_1\in \mathcal{D}_1$, $A_2\in\mathcal{D}_2$, if $f:\pregeometry(A_1)\to \pregeometry(A_2)$ is an isomorphism of pregeometries, and $A_1\strong B_1\in\mathcal{D}_1$, then there exists some $B_1\strong C_1\in \mathcal{D}_1$ and $C_2\in\mathcal{D}_2$ with $A_2\strong C_2$ and an isomorphism $\widehat{f}:\pregeometry(C_1)\to \pregeometry(C_2)$ extending $f$.
\end{itemize}
By a standard back-and-forth argument \cite[Lemma 2.3]{DavidMarcoTwo}, assuming $\emptyset\in \mathcal{D}_1,\mathcal{D}_2$, if $\mathcal{D}_1\overset{*}{\rightsquigarrow}\mathcal{D}_2$ and $\mathcal{D}_2\overset{*}{\rightsquigarrow}\mathcal{D}_1$, then $\pregeometry(\mathbb{D}_1) \cong \pregeometry(\mathbb{D}_2)$, where $\mathbb{D}_i$ is the countable generic structure of $\mathcal{D}_i$.

The following proposition will conclude the proof of our main theorem: 
\begin{prop}
	$\pregeometry(\Mclq)\cong\pregeometry(\Morig_{\tupsize \arity})$
\end{prop}
We split the proof of the proposition between the next three lemmas. Throughout, we think of a relation in $\lang{{\tupsize\arity}}$-structures as an $\tupsize\arity$-tuple as well as an $\arity$-tuple of $\tupsize$-tuples. We begin with a technical lemma: 
\begin{lemma}
\label{remove pathologies}
For every $A\strong B\in \Corig_{\tupsize\arity}$ there exist $B\strong C\in \Corig_{\tupsize\arity}$ and $A\strong D\in \Corig_{\tupsize\arity}$ such that $\pregeometry(D) = \pregeometry(C)$ and $(\bar{b}_{\sigma(1)},\dots,\bar{b}_{\sigma(\arity)})\notin R^D$ for all $(\bar{b}_1,\dots,\bar{b}_{\arity})\in R^D\setminus R^A$ and  $\sigma\in S_{\arity}\setminus \{\mathrm{id}\}$.
\end{lemma}

\begin{proof}
	For each $\bar{a}=(a_1,\dots,a_{\tupsize\arity})\in R^B\setminus R^A$ let $\set{x_{\bar{a}},y_{\bar{a}}}$ be two new elements and let \begin{align*}
	R^C_{\bar{a}} &= \set{(a_1,\dots,a_{\tupsize\arity-2},x_{\bar{a}},y_{\bar{a}}), (a_3,\dots,a_{\tupsize\arity},y_{\bar{a}},x_{\bar{a}})}
	\\
	R^D_{\bar{a}} &= \set{(a_1,\dots,a_{\tupsize\arity-1},x_{\bar{a}}), (a_2,\dots,a_{\tupsize\arity},y_{\bar{a}}),(a_1,a_3,\dots,a_{\tupsize\arity-2},a_{\tupsize\arity},x_{\bar{a}},y_{\bar{a}})}.
	\end{align*}
	Define $C,D$ to be the structures with universe
	\[
	V:=B\cup\bigcup_{\bar{a}\in R^B\setminus R^A}\set{x_{\bar{a}},y_{\bar{a}}}
	\]
	and
	\begin{gather*}
	R^C = R^B \cup \bigcup_{\bar{a}\in R^B\setminus R^A} R^C_{\bar{a}}
	\\
	R^D = R^A \cup \bigcup_{\bar{a}\in R^B\setminus R^A} R^D_{\bar{a}}
	\end{gather*}
	
	We show that $\pregeometry(C) = \pregeometry(D)$.  For $\bar{a}=(a_1,\dots,a_{rs})\in R^B\setminus R^A$ write $S_{\bar{a}} = \set{a_1,\dots,a_{rs},x_{\bar{a}},y_{\bar{a}}}$. Let $X\subseteq V$, it will suffice to show that $X$ is closed in $C$ if and only if it is closed in $D$. If $X$ is closed in $C$ and $\bar{a}\in R^B\setminus R^A$, then $S_{\bar{a}}\subseteq X$ if and only if $|S_{\bar{a}}\cap X| \geq rs-1$. The same holds if we assume $X$ is closed in $D$. So if $X$ is closed in $C$ we immediately get that  $\delta_C(X) = |X\setminus A| - 3\cdot|\setarg{\bar{a}}{S_{\bar{a}}\subseteq X}|+\delta_A(X\cap A)=\delta_D(X)$, implying that $X$ is closed in $D$, since the exact same calculation holds in the opposite direction as well. Indeed, if $Y=\cl^D(X)$ then our argument shows that $\delta_D(Y)=\delta_C(Y)\ge \delta_C(X)=\delta_D(X)$ so $X$ is already closed in $D$. A similar argument shows that if $X$ is closed in $D$ it is also closed in $C$. 
\end{proof}

\begin{lemma}
$(\Corig_{\tupsize \arity}, \strong)\overset{*}{\rightsquigarrow}(\Cclq,\strong)$
\end{lemma}

\begin{proof}
Let $A\in \Corig_{\tupsize\arity}$, $A_c\in\Cclq$ have isomorphic pregeometries. Without loss of generality, assume that $A$ and $A_c$ have the same underlying set and $\pregeometry(A) = \pregeometry(A_c)$. Let $B\in \Corig_{\tupsize\arity}$ be such that $A\strong B$. We have to show that there exist $B\le D\in \Corig_{\tupsize\arity}$ and $A\le C_c\in \Cclq$ such that $\pregeometry(D)\cong_A \pregeometry(C_c)$. Let $D\ge B$ be as provided by Lemma \ref{remove pathologies}. We will now construct $C_c\in \Cclq$ with universe $D$ whose clique-structure captures exactly the $R$-relations not already in $A$: 

For each $\bar{a} = (\bar{a}_1,\dots,\bar{a}_{\arity})\in R^D\setminus R^A$ let $K_{\bar{a}} = \set{\bar{a}_1,\dots,\bar{a}_{\arity}}$. By Remark \ref{ReconstructFromCliques} an $\lang{}$-structure with a given universe can be defined by specifying its maximal cliques. So we define the $\lang{}$-structure $C_c$ with universe $D$ and
\[
\maxcliques{C_c} = \maxcliques{A_c}\cup\setarg{K_{\bar{a}}}{\bar{a}\in R^D	\setminus R^A}.
\]
Observe that for any $L_1,L_2\in\maxcliques{C_c}$ distinct, $|L_1\cap L_2| < s$. If $L_1,L_2\in \maxcliques{A_c}$ then this is by $A_c\in \Cclq$. Otherwise, say if $L_1\notin \maxcliques{C_c}$, by $L_1\nsubseteq L_2$ we have $|L_1\cap L_2|<|L_1| = s$.

For a substructure $X\subseteq D$, denote by $X_c$ the substructure induced by $C_c$ on the universe of $X$. Then for any $X\subseteq D$ we have $\delta(X/X\cap A) = \predim(X_c/X_c\cap A_c)$. Thus, $A_c\strong C_c$, and $\pregeometry(D) = \pregeometry(C_c)$ by \cite[3.5.4.iii]{Mermelstein2016}.
\end{proof}

\begin{lemma}
$(\Cclq, \strong)\overset{*}{\rightsquigarrow}(\Corig_{\tupsize \arity},\strong)$
\end{lemma}

\begin{proof}
Let $A_c\in \Cclq$, $A_{\tupsize\arity}\in \Corig_{\tupsize\arity}$ with $\pregeometry(A_c) = \pregeometry(A_{\tupsize\arity})$ and underlying set $A$. Let $B_c\in \Cclq$ with $A_c\strong B_c$ and underlying set $B$. Fix $f:\maxcliques{B_c}\to {(B^{\tupsize})}^{\arity-1}$ with $f(K) = (\bar{a}_1^K,\dots,\bar{a}_{\arity-1}^K)$ and $E_K = \setcol{\bar{a}^K_i}{1\leq i\leq \arity-1}$ satisfying: 
\begin{enumerate}
\item
$E_K\in [K]^{\arity-1}$
\item 
If $K\cap A^{\tupsize}\in\maxcliques{A_c}$, then $E_K\subseteq (A^{\tupsize})$.
\end{enumerate}
We can require $f$ to be injective, but this is not necessary. Denote by $f(K)\bar{b}$ the concatenation of the tuple $f(K)$ with the tuple $\bar{b}$. Define
\begin{align*}
R_0 &= \setarg{f(K)\bar{b}}{K\in \maxcliques{B_c},K\cap A^{\tupsize}\in\maxcliques{A_c}, \bar{b}\in K\setminus A^{\tupsize}}
\\
R_1 &= \setarg{f(K)\bar{b}}{K\in \maxcliques{B_c},K\cap A^{\tupsize}\notin\maxcliques{A_c}, \bar{b}\in K\setminus E_K}
\end{align*}
and note that each $\bar{a}\bar{b}\in R_0\cup R_1$ has a unique $K\in\maxcliques{B_c}$ such that $\bar{a} = f(K)$ and $\bar{b}\in K$. Let $B_{\tupsize\arity}$ be the structure with underlying set $B$ and
\[
R^{B_{\tupsize\arity}} = R^{A_{\tupsize\arity}}\cup R_0 \cup R_1.
\]

Say that a set $X$ is good if for every $K\in\maxcliques{B_c}$, it holds that $E_K\subseteq X^{\tupsize}$ if and only if $|K\cap X^{\tupsize}|\geq {\arity}$. Then for a good $X$, we have $\predim(X_c/X_c\cap A_c) = \delta(X_{\tupsize\arity}/X_{\tupsize\arity}\cap A)$. Every non-good set $X$ has a good extension $\bar{X}$ with $\delta(\bar{X}/X)<0$, so $A_{\tupsize\arity}\strong B_{\tupsize\arity}$. This finishes the proof, as in the previous lemma.
\end{proof}

\bibliographystyle{plain}
\bibliography{myrefs}
\end{document}